\documentclass[11pt,oneside]{article}
\usepackage{amssymb,amsmath}
\usepackage{xcolor}
\usepackage{enumitem}

\usepackage[russian,english]{babel}
\usepackage[utf8]{inputenc}
\usepackage[T1,T2A]{fontenc}

\usepackage{url}
\renewcommand{\le}{\leqslant}
\renewcommand{\ge}{\geqslant}

\newcommand{\RR}{\mathbb{R}}
\newcommand{\ZZ}{\mathbb{Z}}
\newcommand{\QQ}{\mathbb{Q}}

\newcommand{\NN}{\mathbb{N}}

\newcommand{\PP}{\mathbb{P}}

\newcommand{\VVV}{\mathcal{V}}

\newcommand{\CCC}{\mathcal{C}}
\newcommand{\KKK}{\mathcal{K}}
\newcommand{\FFF}{\mathcal{F}}

\newcommand{\PPP}{\mathcal{P}}
\newcommand{\RRR}{\mathcal{R}}
\newcommand{\vR}{\mathbf{R}}
\newcommand{\MMM}{\mathcal{M}}

\newcommand{\III}{\mathcal{I}}

\newcommand{\vx}{\mathbf{x}}
\newcommand{\vu}{\mathbf{u}}
\newcommand{\vv}{\mathbf{v}}

\newcommand{\vp}{\mathbf{p}}
\newcommand{\vq}{\mathbf{q}}

\newcommand{\va}{\mathbf{a}}

\newtheorem{lemma}{Lemma}
\newtheorem{theorem}{Theorem}
\newtheorem{proposition}{Proposition}
\newtheorem{definition}{Definition}
\newtheorem{remark}{Remark}

\newtheorem{corollary}{Corollary}

\newenvironment{proof}{\textsc{Proof. }}{\ \newline\hspace*{\fill}$\boxtimes$}

\newcommand{\SL}{\mathrm{SL}}

\newcommand{\R}{\mathbb R}
\newcommand{\Z}{\mathbb Z}
\newcommand{\GL}{\operatorname{GL}}
\newcommand{\Disc}{\operatorname{Disc}}
\newcommand{\Hess}{\operatorname{Hess}}
\newcommand{\Ht}{\operatorname{H}}

\begin{document}

\title{Simultaneous Diophantine approximation on the three-dimensional Veronese curve: the complete Hausdorff dimension story}

\author{Dmitry Badziahin \and Nikita Shulga}

\date{}

% \author[Dmitry Badziahin]{Dmitry Badziahin}
% \address{Dzmitry Badziahin, School of Mathematics and Statistics, The University of Sydney, NSW, Australia}
% \email{dzmitry.badziahin@sydney.edu.au}

% \author[Nikita Shulga]{Nikita Shulga}
% \address{Nikita Shulga, Sydney Mathematical Research Institute, The University of Sydney, NSW, Australia}
% \email{nikita.shulga@sydney.edu.au}

%\thanks{The University of Sydney, dzmitry.badziahin@sydney.edu.au}

\maketitle

\begin{abstract}

We determine the Hausdorff dimension of the intersection of the set $\mathcal W_3(\lambda)$ of simultaneously
$\lambda$-well approximable points with the Veronese curve $\VVV_3 \subset \RR^3$
for $\lambda \ge3/5$, thus completing the full range of $\lambda$ values. Precisely, we show that for $\lambda\ge 1/3$,
$$
\dim\bigl(\mathcal W_3(\lambda)\cap\VVV_3\bigr)=
\max\left\{\frac{2-2\lambda}{1+\lambda},
\frac{2}{3(1+\lambda)}\right\}.
$$
To the best of the authors' knowledge, this makes $\VVV_3$ the first
nondegenerate, non-planar curve with a completely determined Hausdorff dimension theory.

\end{abstract}

{\footnotesize{{\em Keywords}: simultaneously well approximable
points, Veronese curve, Hausdorff dimension, simultaneous
Diophantine approximation on manifolds, cubic polynomials of bounded
discriminant

2020 Mathematics Subject Classification: 11J13, 11J54, 11J82, 11K55}}

\section{Introduction}

Let $\MMM\subset\RR^n$ be a smooth $m$-dimensional manifold. A central
problem in metric Diophantine approximation is to understand the points on $\MMM$ that admit simultaneous rational approximations
substantially better than those guaranteed by Dirichlet's theorem. For
$\lambda\ge 1/n$, put
$$
\mathcal W_n(\lambda):=
\left\{\vx\in\RR^n:
\|q\vx-\vp\|_\infty<q^{-\lambda}
\text{ for infinitely many }(q,\vp)\in\ZZ_{>0}\times\ZZ^n
\right\}.
$$
% The general problem considered here is to determine
% \begin{equation}\label{eq:general-dimension-problem}
% \dim\bigl(\MMM\cap\mathcal W_n(\lambda)\bigr),
% \qquad \lambda\ge\frac1n.
% \end{equation}
For so-called nondegenerate manifolds $\MMM$, the induced $m$-dimensional Lebesgue measure of
$\MMM\cap \mathcal W_n(\lambda)$ is understood through the work of Kleinbock
and Margulis~\cite{kle_mar_1998}, who proved the Sprind\v{z}uk conjecture.
%in particular the extremality of nondegenerate manifolds;
In the present notation, their result implies that $\MMM\cap\mathcal W_n(\lambda)$ has measure zero
for every $\lambda>1/n$. For general monotone approximation functions,
Beresnevich and Yang~\cite{ber_yan_2023} proved the convergence part of
Khintchine's theorem for arbitrary nondegenerate manifolds in $\RR^n$. Thus, beyond
the Dirichlet threshold, Hausdorff dimension is the natural measure of the size of the corresponding set. Here a manifold $\MMM\in \RR^n$ is called nondegenerate if the first several partial quotients of its parametrisation span the whole space $\RR^n$ at every point. For an analytic manifolds this condition essentially means that they do not belong to any proper linear subspace of $\RR^n$.

The dimension problem is considerably more delicate. The usual
``independence'' heuristic predicts
\begin{equation}\label{dim_heur}
\dim\bigl(\MMM\cap\mathcal W_n(\lambda)\bigr)
=
\frac{n+1}{1+\lambda}-(n-m).
\end{equation}
Beresnevich and Yang~\cite{ber_yan_2023} verified this equation in an
explicit neighbourhood of the Dirichlet exponent, and Schindler, Srivastava
and Technau~\cite{sst_2024} subsequently enlarged its size. In particular,
for a nondegenerate curve $\CCC\subset\RR^n$, their result gives
\begin{equation}\label{eq:sst-curve-range}
\dim\bigl(\CCC\cap\mathcal W_n(\lambda)\bigr)
=
\frac{2-(n-1)\lambda}{1+\lambda}
\end{equation}
whenever $\lambda$ lies within the interval
$$
\frac1n\le\lambda<
\frac{2n^2+n+2}{(n^2+n+1)(2n-1)}
=
\frac1n+
\frac{n+1}{n(2n-1)(n^2+n+1)}.
$$
Notice that its length is $\frac{1}{2n^3}+O(n^{-4})$.

Lower bounds are known in substantially wider ranges. Beresnevich, Lee,
Vaughan and Velani~\cite{blvv_2017} proved, for every $C^2$
$m$-dimensional manifold $\MMM\subset\RR^n$, that
\begin{equation}\label{eq:blvv-lower}
\dim\bigl(\MMM\cap\mathcal W_n(\lambda)\bigr)
\ge \frac{n+1}{1+\lambda}-(n-m),
\qquad
\frac1n\le\lambda<\frac{1}{n-m}.
\end{equation}
For analytic nondegenerate curves, Beresnevich~\cite{beresnevich_2012} proved a Jarn\'ik-type
divergence theorem which, in particular, yields the lower bound
\begin{equation}\label{eq:bvvz-curve-range}
\dim\bigl(\CCC\cap\mathcal W_n(\lambda)\bigr)
\ge \frac{2-(n-1)\lambda}{1+\lambda},
\qquad
\frac1n\le\lambda<\frac{3}{2n-1}.
\end{equation}
Beresnevich, Vaughan, Velani and
Zorin~\cite{bvvz_2021} subsequently removed the analyticity assumption and
proved the corresponding divergence theorem for arbitrary nondegenerate
curves. Writing $\tau_{n,1}$ for the supremum of the exponents up to which
\eqref{dim_heur} is expected to hold uniformly for nondegenerate curves in
$\RR^n$, Beresnevich and Yang~\cite{ber_yan_2023} conjectured that
\begin{equation}\label{eq:tau-curve-conjecture}
\tau_{n,1}=\frac{3}{2n-1}.
\end{equation}
% For curves in $\RR^3$, this threshold is $3/5$. Thus the general curve theory
% already provides the expected lower bound on $\VVV_3$ throughout
% $1/3\le\lambda<3/5$. The main difficulty is to obtain matching upper bounds
% and to determine what happens beyond this threshold.

The most natural model family is the Veronese curves
$$
\VVV_n:=\{(\xi,\xi^2,\ldots,\xi^n):\xi\in\RR\}\subset\RR^n,
$$
also commonly called the \emph{moment curves}. They are the simplest
polynomial examples of \emph{nondegenerate curves}; indeed, the first $n$ derivatives of their parametrisation span $\RR^n$ at every point. Consequently, they form a
canonical testing ground for the general theory, while their arithmetic
structure allows phenomena that are not captured by purely differential-geometric considerations.

Given
$\vp=(p_0,p_1,\ldots,p_n)\in\ZZ_{>0}\times\ZZ^n$, consider the inequality
\begin{equation}\label{main_ineq}
\max_{1\le i\le n}|p_0\xi^i-p_i|<p_0^{-\lambda},
\end{equation}
and define
$$
S_n(\lambda):=
\left\{\xi\in\RR:\eqref{main_ineq}
\text{ holds for infinitely many }
\vp\in\ZZ_{>0}\times\ZZ^n
\right\}.
$$
The set $S_n(\lambda)$ is the projection of $\VVV_n\cap\mathcal W_n(\lambda)$ onto its first coordinate. Hence, on every bounded interval, its Hausdorff dimension agrees
with that of $\VVV_n\cap\mathcal W_n(\lambda)$. For convenience, we will often work with one-dimensional $S_n(\lambda)$ instead of the points on the curve $\VVV_n$ itself.

% Thus $S_n(\lambda)$ is the parameter set of $\lambda$-well approximable
% points on $\VVV_n$; on every bounded interval its Hausdorff dimension agrees
% with that of $\VVV_n\cap\mathcal W_n(\lambda)$.  Sprind\v{z}uk's theorem~\cite{sprindzuk_1969}, together with Khintchine's transference principle, implies that the Lebesgue measure of $S_n(\lambda)$ is zero for all $\lambda>1/n$.  The problem is therefore to determine
% $$
% \dim S_n(\lambda),\qquad \lambda\ge\frac1n.
% $$

Even for this model family, complete results about $\dim (S_n(\lambda))$ are rare. For the parabola
$\VVV_2$, Beresnevich, Dickinson and Velani~\cite{bdv_2007} (with an
appendix by Vaughan) established the Hausdorff theory in the critical range,
while Budarina, Dickinson and Levesley~\cite{bdl_2010} treated sufficiently
small approximation functions. Together, these results give the complete
formula
$$
\dim S_2(\lambda)
= \max\left\{ \frac{2-\lambda}{1+\lambda}, \frac{1}{1+\lambda}\right\},\quad \lambda\ge\frac12.
% \begin{cases}
% \displaystyle\frac{2-\lambda}{1+\lambda},&\displaystyle\frac12\le\lambda\le1,\\[5pt]
% \displaystyle\frac{1}{1+\lambda},&\displaystyle\lambda\ge1.
% \end{cases}
$$
In higher dimensions, only partial ranges were previously known. The first author~\cite{badziahin_2025}
proved that
\begin{equation}\label{dim_ver}
\dim S_n(\lambda)=\frac{2-(n-1)\lambda}{1+\lambda} \quad\text{ for } \quad \frac{1}{n} \le\lambda\le \frac{2}{2n-1}
\end{equation} and obtained the wider range
$1/3\le\lambda\le1/2$ when $n=3$. In subsequent work~\cite{badziahin_2026}, the cubic range in~\eqref{dim_ver} was extended to
$$
\frac13\le\lambda\le\frac35.
$$
At the opposite end of the approximation range, Badziahin and
Bugeaud~\cite{bad_bug_2020} proved that
$$
\dim S_n(\lambda)=\frac{2}{n(1+\lambda)}
\qquad\text{for }\lambda\ge\frac{n+4}{3n}.
$$
Thus, for the cubic Veronese curve $\VVV_3$, the Hausdorff dimension of $S_3(\lambda)$ remained unknown on the interval $\frac35<\lambda<\frac79$.
%remained between the two ranges in which the exact dimension was known.

Our main result closes this gap and determines the Hausdorff dimension of the
well-approximable points on $\VVV_3$ for every approximation exponent. To
the best of our knowledge, this yields the first complete dimension law for a
nondegenerate, non-planar curve.

\begin{theorem}\label{th1}
For every $\lambda\ge \frac13$,
$$
\dim S_3(\lambda)
= \max\left\{\frac{2-2\lambda}{1+\lambda}, \frac{2}{3(1+\lambda)}\right\}.
% \begin{cases}
% \displaystyle \frac{2-2\lambda}{1+\lambda},
% & \displaystyle \frac13\le\lambda\le\frac23,\\[6pt]
% \displaystyle \frac{2}{3(1+\lambda)},
% & \displaystyle \lambda\ge\frac23.
% \end{cases}
$$
\end{theorem}

Theorem~\ref{th1} exhibits a phase transition at $\lambda=2/3$. Up to this
point, the dimension agrees with the value predicted by the general manifold
heuristic~\eqref{dim_heur} for a curve in $\RR^3$; beyond $2/3$, a special
arithmetic concentration of rational points along tangent subspaces becomes
dominant. This is particularly notable in view of the range for $\lambda$ in~\eqref{eq:bvvz-curve-range}.
The general curve theory provides the generic lower bound for
$1/3\le\lambda<3/5$, and Badziahin~\cite{badziahin_2026} proved the matching
upper bound through $\lambda=3/5$, whereas Theorem~\ref{th1} shows that the
same dimension law persists all the way to $2/3$. Thus the conjectural uniform
threshold $3/5$ in~\eqref{eq:tau-curve-conjecture} is not a transition point for this curve. Our theorem fills the previously open interval
$3/5<\lambda<7/9$ and identifies the mechanism by which the Veronese curve
eventually ceases to follow the generic codimension law.
% To the best of our
% knowledge, $\VVV_3$ is the first nondegenerate, non-planar curve for which the
% Hausdorff dimension of the simultaneously well-approximable points is known
% throughout the full range $\lambda\ge1/3$.

A second theme of the paper is a closely related problem of counting
rational points near manifolds. For a bounded $m$-dimensional manifold
$\MMM\subset\RR^n$, $Q\ge2$ and $0<\delta<1/2$, define
$$
N_{\MMM}(Q,\delta)
:=\#\left\{(q,\vp)\in\ZZ^{n+1}:1\le q\le Q,
\inf_{\vx\in\MMM}\|q\vx-\vp\|_\infty\le\delta\right\}.
$$
This counting problem for $N_{\MMM}(Q,\delta)$ has a long history. The random-point heuristic suggests
\begin{equation}\label{nm_heur}
N_{\MMM}(Q,\delta)\ll Q^{m+1}\delta^{n-m}.
\end{equation}
%for a suitable error term $E(Q,\delta)$.
 Huxley~\cite{huxley_1994} obtained an early fundamental result for planar curves,
and Vaughan and Velani~\cite{vau_vel_2006} later proved, for compact $C^3$
planar curves $\CCC$ with non-vanishing curvature, the essentially optimal
estimate
$$
N_{\CCC}(Q,\delta)\ll Q^2\delta+Q^{1+\varepsilon}.
$$
In higher dimensions, Beresnevich, Vaughan, Velani and
Zorin~\cite{bvvz_2017} obtained essentially best possible upper bounds under
their dimensional and nondegeneracy hypotheses. Huang~\cite{huang_2020}
proved a sharp asymptotic formula for rational points near hypersurfaces with
non-vanishing Gaussian curvature, while Srivastava~\cite{srivastava_2025}
treated non-isotropic neighbourhoods of higher-codimension manifolds under a
strong curvature condition. These results illustrate how sensitive the counting problem is to the
geometry of the manifold.

For the Veronese curves, the heuristic estimate~\eqref{nm_heur} cannot hold
uniformly over all rational points. Rational subspaces tangent to the curve
may contain an anomalously large number of nearby integer points. For example, points of the form
$$
(q,p,0,\ldots,0),\qquad Q/2\le |q|\le Q,
\qquad |p|\le (Q\delta/2)^{1/2},
$$
already contribute
\begin{equation}\label{eq:tangent-obstruction}
N_{\VVV_n}(Q,\delta)\gg Q^{3/2}\delta^{1/2}
>Q^2\delta^{n-1}
\qquad\text{whenever }\delta<Q^{-1/(2n-3)}.
\end{equation}
For the cubic curve $\VVV_3$, setting $\delta=Q^{-\lambda}$ shows that the tangent contribution exceeds the heuristic count for every $\lambda>1/3$, that is, throughout the entire
nontrivial Diophantine range. The main observation of the present paper is
that for $\VVV_3$, all rational points responsible for this excess can be
geometrically isolated.

Given $\va=(a_0,\ldots,a_m)\in\ZZ^{m+1}\setminus\{\mathbf 0\}$, $m\le n$, let $L_\va$ denote the
$m$-dimensional homogeneous subspace of $\RR^{n+1}$ defined by
\begin{equation}\label{def_la}
a_0x_j+a_1x_{j+1}+\cdots+a_mx_{j+m}=0,
\qquad 0\le j\le n-m.
\end{equation}
We identify $\va$ with the polynomial
$P_\va(x)=a_0+a_1x+\cdots+a_mx^m$ and also write $L_P$ for $L_\va$.
For a reduced rational number $u/v$, put
$$
P_{u,v}(x):=(vx-u)^2.
$$

\begin{definition}
A point $\vp=(p_0,p_1,p_2,p_3)\in\ZZ^4$ is called \textbf{special} if
$\vp\in L_{P_{u,v}}$ for some coprime $(u,v)\in\ZZ^2$. All other integer
points are called \textbf{generic}.
\end{definition}

Each $L_{P_{u,v}}$ is a two-dimensional homogeneous subspace associated
with the tangent line to $\VVV_3$ at the rational point with parameter
$u/v$. Once these tangent configurations are removed, the heuristic
counting bound is restored without loss. More precisely, fix a compact
interval $I\subset\RR$ and, in the counting statements, write
$$
\VVV_3=\{(\xi,\xi^2,\xi^3):\xi\in I\}.
$$
For $Q\ge2$ and $0<\delta<1/2$, define
\begin{align*}
N_{g,\VVV_3}(Q,\delta)
&:=\#P_g(Q,\delta)\\
&:=\#\bigg\{\vp\in\ZZ^4:\ 1\le p_0\le Q,\quad \vp\text{ is generic},
\inf_{\vx\in\VVV_3}\max_{1\le i\le3}|p_0x_i-p_i|\le\delta\bigg\}.
\end{align*}
The dependence on $I$ is suppressed in the notation. The next theorem is our main counting input, and it agrees with the heuristic~\eqref{nm_heur}.

\begin{theorem}\label{th2}
Let $I\subset\RR$ be compact. Uniformly for $Q\ge2$ and
$0<\delta<1/2$,
\begin{equation}\label{th2_eq}
N_{g,\VVV_3}(Q,\delta)\ll_I Q^2\delta^2.
\end{equation}
\end{theorem}

% Its proof exploits a feature specific to the
% cubic Veronese curve.  To
% $\vp=(p_0,p_1,p_2,p_3)\in\ZZ^4$ we associate the binomially normalised binary
% cubic
% $$
% F_\vp(X,Y)=p_0X^3+3p_1X^2Y+3p_2XY^2+p_3Y^3.
% $$
% The special points are precisely those for which $F_\vp$ has zero
% discriminant, equivalently those for which the Hessian covariant is
% degenerate; see Proposition~\ref{prop:special-hessian-correspondence}.
% Moreover, proximity to $\VVV_3$ imposes strong restrictions on the Hessian
% coefficients; see Lemma~\ref{lem:approx-window}.  The counting problem is
% thereby reduced to a lattice problem for nondegenerate binary cubics with
% restricted Hessian coefficients.  Establishing the resulting uniform bound
% is the principal technical part of the proof of Theorem~\ref{th2}.

\begin{remark}
Very recently, Chen et al.~\cite{chen_seeger_srivastava_technau_2026}
proved counting estimates for rational points near space curves in $\RR^3$.
Applied to the cubic Veronese curve, their result gives
$$
N_{\VVV_3}(Q,\delta)\ll_\varepsilon
Q^2\delta^2+Q^{4/3+\varepsilon}.
$$
For $\delta=Q^{-\lambda}$ and $\lambda\ge 1/3$, the error term
dominates and in view of the remark~\eqref{eq:tangent-obstruction}, rational points in the tangent subspace of $\VVV_3$ at zero already contribute more than $Q^2\delta^2$. This is why we split rational points into two families: generic and special.  Theorem~\ref{th2} shows that for generic points the main term dominates throughout the whole range of $\delta$.
% Very recently, Chen et.al.~\cite{chen_seeger_srivastava_technau_2026}
% proved counting estimates for rational points near space curves in $\RR^3$.
% Applied to the cubic Veronese curve, their result gives
% $$
% N_{\VVV_3}(Q,\delta)\ll_\varepsilon
% Q^2\delta^2+Q^{4/3+\varepsilon}.
% $$
% For $\delta=Q^{-\lambda}$ and any fixed $\lambda\ge1/3$, $Q^2\delta^2$
% is not the main term, which is expected in view of the remark~\eqref{eq:tangent-obstruction}. This is the exact reason we split rational points into two families: generic and special. After excluding the special tangent configuration, Theorem~\ref{th2}
% recovers the heuristic bound $Q^2\delta^2$ for the generic points in the range $\lambda\ge1/3$
% relevant for the Hausdorff dimension problem.
\end{remark}

Theorem~\ref{th2} immediately yields the required Hausdorff-dimension
upper bound for the generic contribution. For an interval $I\subset\RR$,
write $S_3(I,\lambda)=S_3(\lambda)\cap I$ and set
$$
S_g(I,\lambda):=\left\{x\in I:
\max_{1\le i\le3}|p_0x^i-p_i|<p_0^{-\lambda}
\text{ for infinitely many generic }\vp\in\ZZ_{>0}\times\ZZ^3\right\}.
$$

\begin{corollary}\label{cor:generic-upper}
For every interval $I\subset\RR$ and every $\lambda\ge\frac13$,
\begin{equation}\label{eq:generic-upper}
\dim S_g(I,\lambda)
\le
\max\left\{\frac{2-2\lambda}{1+\lambda},0\right\}.
\end{equation}
\end{corollary}

Special points require a different argument. Although they are too
numerous to be covered by Theorem~\ref{th2}, their location on the tangent subspaces
$L_{P_{u,v}}$ provides an explicit parametrisation, allowing the corresponding limsup set to be analysed directly. Let
$$
S_s(\lambda):=
\left\{\xi\in\RR:
\max_{1\le i\le3}|p_0\xi^i-p_i|<p_0^{-\lambda}
\text{ for infinitely many special }\vp\in\ZZ_{>0}\times\ZZ^3\right\}.
$$

\begin{theorem}\label{th3}
For $\lambda\ge \frac35$ one has
$$
\dim S_s(\lambda)
=
\max\left\{\frac{2-2\lambda}{1+\lambda},
\frac{2}{3(1+\lambda)}\right\}.
$$
\end{theorem}

Notice that Theorem~\ref{th3} together with Corollary~\ref{cor:generic-upper} immediately implies Theorem~\ref{th1} for $\lambda\ge \frac35$. Indeed, for this range the special tangent configurations already account for the full dimension appearing in
Theorem~\ref{th1}, while Theorem~\ref{th2} shows that generic points contribute no larger dimension. Smaller values of $\lambda$ are already covered in previous works~\cite{badziahin_2025, badziahin_2026} hence we do not pursue an analogue of Theorem~\ref{th3} below $3/5$.

% Thus,  Combined with the previously known formula
% \eqref{dim_ver} for $1/3\le\lambda\le3/5$, this proves the main theorem.
% We do not pursue an analogue of Theorem~\ref{th3} below $3/5$, since the
% dimension of the full set is already known there from the work of
% Badziahin~\cite{badziahin_2026}.

We finally mention the equivalent interpretation of $S_n(\lambda)$ in terms of Diophantine
exponents. Bugeaud and Laurent~\cite{bug_lau_2005} introduced the exponent
$\lambda_n(\xi)$ as the supremum of the exponents $\lambda$ for which
\eqref{main_ineq} has infinitely many solutions; see also the survey of
Bugeaud~\cite{bugeaud_2016} and the recent work of
Badziahin~\cite{badziahin_2025,badziahin_2026}. With this notation,
$$\{\xi:\lambda_n(\xi)>\lambda\}\subseteq S_n(\lambda)\subseteq \{\xi:\lambda_n(\xi)\ge\lambda\},$$ so Theorem~\ref{th1}
determines the complete upper-level dimension function for $\lambda_3$.
Bugeaud and Laurent~\cite{bug_lau_2007} posed the problem of whether the spectrum of
$\lambda_n$ is the whole interval $[1/n,\infty]$. Theorem~\ref{th1} answers this positively for $n=3$. We will not use the exponent formulation elsewhere in the paper.

The paper is organised as follows. Section~\ref{sec:binary-cubics}
introduces the algebraic and geometric framework for the proof; in particular,
it identifies the special tangent subspaces with binary cubic forms of zero discriminant. In
Section~\ref{sec:window} we translate proximity to $\VVV_3$ into quantitative
bounds for the Hessian coefficients and prove the counting estimate of
Theorem~\ref{th2}. Section~\ref{sec:upper} applies this counting theorem to obtain Corollary~\ref{cor:generic-upper}.
Section~\ref{sec:special-lower} is devoted to the special points and proves the lower bound in
Theorem~\ref{th3} by constructing suitable special approximants and applying
a Cantor-type construction together with the mass distribution principle.
Section~\ref{sec:special-upper} proves the matching upper bound for the special set using the
explicit parametrisation of the lattices $L_{P_{u,v}}$. These ingredients
complete the proof of Theorem~\ref{th1}.

\section{Binary cubics, Hessians, and changes of variables}
\label{sec:binary-cubics}

We work with the lattice of binomially normalised binary cubics
$$
F(X,Y)=aX^3+3bX^2Y+3cXY^2+dY^3,
\qquad (a,b,c,d)\in\Z^4.
$$
For a point $\mathbf{p}=(p_0,p_1,p_2,p_3) \in \Z^4$, we associate the cubic $F=F_{\mathbf p}$ with the coefficients $(a,b,c,d)=(p_0,p_1,p_2,p_3)$. For such a binomially normalised cubic $F$ we define $\Ht(F)$ to be the maximum of
$|a|,|b|,|c|,|d|$. On the other hand, for quadratic or linear forms $F$, $\Ht(F)$ denotes the
maximum of the absolute values of its coefficients in the standard monomial
basis.

\subsection{The Hessian covariant}
We write $\Disc(F)$ for the discriminant of a binary form. In the proof we will be using it for both quadratic and cubic binary forms.

Define the normalised Hessian of $F$ by
\begin{equation}\label{eq:hessian-normalized}
\Hess(F)=\frac1{36}\left(F_{XX}F_{YY}-F_{XY}^2\right).
\end{equation}
Since
$$
F_{XX}=6(aX+bY),\quad
F_{XY}=6(bX+cY),\quad
F_{YY}=6(cX+dY),
$$
we obtain
\begin{equation}\label{eq:hessian-explicit}
\Hess(F)=\Delta_0X^2+\Delta_1XY+\Delta_2Y^2,
\end{equation}
where
\begin{equation}\label{eq:deltas-general}
\Delta_0=ac-b^2,
\qquad
\Delta_1=ad-bc,
\qquad
\Delta_2=bd-c^2.
\end{equation}
For $F=F_{\mathbf p}$, this becomes
\begin{equation}\label{eq:minors-def}
\Delta_0=p_0p_2-p_1^2,\qquad
\Delta_1=p_0p_3-p_1p_2,\qquad
\Delta_2=p_1p_3-p_2^2.
\end{equation}
These are precisely the three Hessian coefficients attached to the rational approximation vector $\mathbf p$.

The discriminant of the binary quadratic form $\Hess(F)$ is
\begin{equation}\label{eq:hess-disc}
\Disc\bigl(\Hess(F)\bigr)=\Delta_1^2-4\Delta_0\Delta_2.
\end{equation}
One can check that the usual discriminant of the binary cubic satisfies
\begin{equation}\label{eq:cubic-hess-disc}
\Disc(F)=-27\,\Disc\bigl(\Hess(F)\bigr).
\end{equation}
% Indeed, substituting \eqref{eq:deltas-general} into \eqref{eq:hess-disc} gives
% $$
% \Disc\bigl(\Hess(F)\bigr)=a^2d^2-6abcd+4ac^3+4b^3d-3b^2c^2,
% $$
% whereas the discriminant of $aX^3+3bX^2Y+3cXY^2+dY^3$ is
% $$
% -27\left(a^2d^2-6abcd+4ac^3+4b^3d-3b^2c^2\right).
% $$

%\subsection{Change of variables and Hessian covariance}%DO WE NEED THIS SUBSECTION? WE CAN JUST SAY THAT HESSIAN IS A COVARIANT IS A COMMON KNOWLEDGE AND PERHAPS REFER TO SOME TEXTBOOK

For
$$
g=\begin{pmatrix}\alpha&\beta\\ \gamma&\delta\end{pmatrix}
\in\GL_2(\mathbb R),
$$
define the change of variables
\begin{equation}\label{eq:gl2-action}
(g\cdot F)(X,Y)
=
F(\alpha X+\beta Y,\gamma X+\delta Y).
\end{equation}
If $g\in\GL_2(\mathbb Z)$ then this substitution sends the lattice of integral
binomially normalised cubics to itself. Since $g^{-1}\in\GL_2(\mathbb Z)$ as well, it acts as an automorphism of
this lattice.

% Write
% $$
% X'=\alpha X+\beta Y,\qquad Y'=\gamma X+\delta Y.
% $$
% Since this change of variables is linear, the multivariable chain rule gives
% $$
% \begin{pmatrix}
% (g\cdot F)_{XX}(X,Y) & (g\cdot F)_{XY}(X,Y)\\
% (g\cdot F)_{XY}(X,Y) & (g\cdot F)_{YY}(X,Y)
% \end{pmatrix}
% =
% g^{\mathsf T}
% \begin{pmatrix}
% F_{XX}(X',Y') & F_{XY}(X',Y')\\
% F_{XY}(X',Y') & F_{YY}(X',Y')
% \end{pmatrix}
% g.
% $$
% Taking determinants gives
% $$
% \det
% \begin{pmatrix}
% (g\cdot F)_{XX} & (g\cdot F)_{XY}\\
% (g\cdot F)_{XY} & (g\cdot F)_{YY}
% \end{pmatrix}
% =
% (\det g)^2
% \det
% \begin{pmatrix}
% F_{XX}(X',Y') & F_{XY}(X',Y')\\
% F_{XY}(X',Y') & F_{YY}(X',Y')
% \end{pmatrix}.
% $$
Using the normalisation in \eqref{eq:hessian-normalized}, one can calculate
\begin{equation}\label{eq:hessian-covariance}
\Hess(g\cdot F)=(\det g)^2\,g\cdot\Hess(F),
\end{equation}
where the action on the right is the same change of variables applied to
the quadratic form $\Hess(F)$. Explicitly,
$$
(g\cdot\Hess(F))(X,Y)
=
\Hess(F)(\alpha X+\beta Y,\gamma X+\delta Y).
$$

%(maybe some general info on Hessians, as they are often used in context of cubics)

\subsection{Connecting geometric and algebraic views of special $\vp$}
We say that a nonzero vector $\vp\in\RR^4$ lies on the Veronese curve $\VVV_3$ if its
projective class $[\vp]\in\PP^3(\RR)$ lies on the projective closure of
$\{[1:\xi:\xi^2:\xi^3]:\xi\in\RR\}$.

%Let
%$$
%F_\vp(X,Y)=p_0X^3+3p_1X^2Y+3p_2XY^2+p_3Y^3
%$$
%be the binary cubic associated with
%$\vp=(p_0,p_1,p_2,p_3)\in\ZZ^4\setminus\{\mathbf 0\}$.

% Consider the geometric meaning of pure cubes and repeated linear
% factors. If $L=vX+uY$, then

% $$
% L^3=v^3X^3+3v^2uX^2Y+3vu^2XY^2+u^3Y^3.
% $$
% Thus the pure cube $L^3$ corresponds to the point
% $(v^3,v^2u,vu^2,u^3)$ on the Veronese curve.

% Now choose a rational linear form $M$ which is not proportional to $L$.
% For the family $F(t)=(L+tM)^3$, one has
% $F(0)=L^3$ and $F'(0)=3L^2M$. Hence the tangent direction at the
% Veronese point corresponding to $L^3$ is $L^2M$, and the associated
% two-dimensional tangent subspace is

% $$
% L_{P_{u,v}} = \operatorname{span}_{\RR}\{L^3,L^2M\}=L^2\RR[X,Y]_1.
% $$
% Consequently, a binary cubic lies in this tangent subspace if and only if
% it is divisible by $L^2$. It represents the point of tangency itself if
% and only if it is a scalar multiple of $L^3$.

% We can now state the precise correspondence between the geometry of the
% special subspaces and the algebraic invariants of $F_\vp$.

First, let us state a trivial, yet useful fact about the discriminant.

\begin{lemma}\label{lem:special-discriminant}
Let $F\in\QQ[X,Y]$ be a nonzero binary cubic form. Then
$$
\Disc(F)=0 \quad\Longleftrightarrow\quad F=L^2N
$$
for some linear forms $L,N\in\QQ[X,Y]$.
\end{lemma}

\begin{proof}
Over $\overline{\QQ}$, the discriminant of a binary cubic
vanishes if and only if the cubic has a repeated projective root.
Equivalently, $F=L^2N$ for some linear forms $L,N$ over $\overline{\QQ}$.

It remains only to check that the repeated factor is defined over
$\QQ$. A nonzero cubic has at most one repeated projective
root. Since $F$ has rational coefficients, every Galois conjugate of
a repeated root is again a repeated root. An irrational repeated root is impossible for a cubic, since two distinct repeated roots would have total multiplicity at least $4$. Hence, the repeated root is rational, and then $N=F/L^2$ also belongs to $\QQ[X,Y]$.

The converse is immediate, since a cubic divisible by $L^2$ has a
repeated root.
\end{proof}

Now we prove the connection between the generic--special points definition and algebraic objects.

\begin{proposition}
\label{prop:special-hessian-correspondence}
Let $\vp=(p_0,p_1,p_2,p_3)\in\ZZ^4\setminus\{\mathbf 0\}$. For coprime
integers $u,v$, put

$$
\mathbf q(u,v)=(v^3,v^2u,vu^2,u^3).
$$

Then the following statements hold.

\begin{enumerate}[label=\rm(\roman*)]

\item The point $\vp$ is generic if and only if $\Disc(\Hess(F_\vp))\neq0$ or, equivalently, $ \Disc(F_\vp)\ne0$.

\item The point $\vp$ is special and lies in a tangent subspace away from
its point of tangency, that is, for some coprime integers $u,v$,

$$
\vp\in L_{P_{u,v}}\setminus\QQ\mathbf q(u,v),
$$

if and only if $\Disc(F_\vp)=0$ and $\Hess(F_\vp)\not\equiv0$. In this case, the pair $(u,v)$ is uniquely defined by $\vp$ up to replacing $(u,v)$ by $(-u,-v)$.

\item The point $\vp$ lies on the Veronese curve, that is, for some coprime
integers $u,v$,

$$
\vp\in \QQ\mathbf q(u,v),
$$

if and only if $\Hess(F_\vp) \equiv 0$.
Equivalently, this algebraic condition may be written as

$$
\Disc(F_\vp)=0
\qquad\text{and}\qquad
\Hess(F_\vp)\equiv0.
$$

\end{enumerate}

In particular, if $\vp$ is primitive, $p_0>0$, and case~\textup{(iii)} holds,
then $\vp=(v^3,v^2u,vu^2,u^3)$ for coprime integers $u,v$ with $v>0$.
\end{proposition}

\begin{proof}
% We have already proved that

% $$
% \vp\text{ is special}
% \quad\Longleftrightarrow\quad
% \Disc(F_\vp)=0.
% $$

By definition, $\vp\in L_{P_{u,v}}$ if and only if
$$
u^2p_0-2uvp_1+v^2p_2=u^2p_1-2uvp_2+v^2p_3=0.
$$
Let
$$
Q_1(X,Y):=p_0X^2+2p_1XY+p_2Y^2,\qquad Q_2(X,Y):=p_1X^2+2p_2XY+p_3Y^2.
$$
The above equations are equivalent to
$$
Q_1(u,-v)=Q_2(u,-v)=0,
$$
which is equivalent to the fact that the forms $Q_1$ and $Q_2$ are
multiples of $L=vX+uY$. Next, one computes
$$
F_\vp=XQ_1+YQ_2,\qquad (F_{\vp})_X=3Q_1,\qquad (F_{\vp})_Y=3Q_2.
$$
Therefore $L\mid F_\vp,(F_{\vp})_X,(F_{\vp})_Y$, and hence $L^2\mid F_\vp$.
By Lemma~\ref{lem:special-discriminant}, the last condition implies $\Disc(F_\vp)=0$. To complete the statement~(i), we notice that the family of binary forms $\{F_{\vp}: \vp\in L_{P_{u,v}}\}$ is a two-dimensional linear subspace of the space of binary cubic forms. The same is true for the subspace $\FFF_{u,v}$ of multiples of $L^2$. Since one subspace is contained in the other and their dimensions coincide, the subspaces coincide too. Thus, varying $(u,v)$ and applying Lemma~\ref{lem:special-discriminant}, we obtain
$$
\vp\text{ is special}\quad\Longleftrightarrow\quad \Disc(F_\vp)=0.
$$
Taking complements proves~\rm(i).

%By Lemma~\ref{lem:special-discriminant}, we have that $\vp$ is special if and only if $\Disc(F_{\vp})=0$. Indeed, if $\vp$ is special, then $\vp\in L_{P_{u,v}}$ for some coprime integers $u,v$, so $F_{\vp}$ is divisible by $(vX+uY)^2$ and therefore has zero discriminant. Conversely, if $\Disc(F_{\vp})=0$, then $F_{\vp}=L^2N$ for some rational linear forms $L,N$. Write $L=vX+uY$ for some coprime integers $u,v$. Then by description of $L_{P_{u,v}}$ from the beginning of the subsection, we get $\vp\in L_{P_{u,v}}$, so $\vp$ is special.

%Taking complements proves \rm(i). The equivalence with
%$\Disc\Hess(F_\vp)\ne0$ follows from \eqref{eq:cubic-hess-disc}.

Now, suppose that $\vp$ is special, which implies that
the cubic $F_\vp$ is divisible by $L^2$. Hence $F_\vp=L^2N$ for some rational linear form $N$. Choose a rational linear form $M=rX+sY$ with $vs-ur=1$, which is linearly independent of $L$.
Writing $N=\omega L+3\mu M$, we obtain

$$
F_\vp=\omega L^3+3\mu L^2M.
$$

A direct calculation gives

$$
\Hess(F_\vp)=-\mu^2 (vs-ur)^2L^2 = -\mu^2 L^2.
$$

Therefore

\begin{equation}\label{eq:equivalence1}
\Hess(F_\vp)\equiv 0\quad\Longleftrightarrow\quad\mu=0\quad\Longleftrightarrow\quad F_\vp=\omega L^3.
\end{equation}

The binomially normalised coefficient vector of $L^3=(vX+uY)^3$ is
$\mathbf q(u,v)$. Consequently,

\begin{equation}\label{eq:equivalence2}
F_\vp=\omega L^3
\quad\Longleftrightarrow\quad
\vp=\omega\mathbf q(u,v)
\quad\Longleftrightarrow\quad
\vp\in\QQ\mathbf q(u,v).
\end{equation}

%This proves \rm(iii).

% If $\Hess(F_\vp)\not\equiv0$, then $\mu\ne0$. Thus $F_\vp$ is divisible by
% $L^2$ but is not a scalar multiple of $L^3$. In terms of coefficient
% vectors, this is precisely

% $$
% \vp\in L_{P_{u,v}}\setminus\QQ\mathbf q(u,v).
% $$

% Since $\vp$ is special, one also has $\Disc(F_\vp)=0$. This proves both
% directions of \rm(ii).

If $\Hess(F_\vp)\equiv0$, then $\Disc\bigl(\Hess(F_\vp)\bigr)=0$, so
\eqref{eq:cubic-hess-disc} gives $\Disc(F_\vp)=0$. Thus $\vp$ is
special, and the equivalences \eqref{eq:equivalence1} and \eqref{eq:equivalence2} prove \rm(iii).

For a special point, the same equivalences show that
$$
\Hess(F_\vp)\not\equiv0
\quad\Longleftrightarrow\quad
\vp\in L_{P_{u,v}}\setminus\QQ\mathbf q(u,v),
$$
which proves \rm(ii).

In case \rm(ii), the pair $(u,v)$ is unique up to simultaneous change of
sign. Indeed, $L$ is the repeated linear factor of $F_\vp$. If a nonzero
cubic had two non-proportional repeated linear factors $L_1$ and $L_2$,
then $L_1^2L_2^2$ would divide it, which is impossible because this product
has degree $4$.

Finally, suppose that $\vp$ is primitive, $p_0>0$, and
$\vp=c\mathbf q(u,v)$ with $u,v$ coprime. The vector $\mathbf q(u,v)$ is
primitive, so the integrality and primitivity of $\vp$ imply $c=\pm1$.
Replacing $(u,v)$ by $(-u,-v)$ absorbs the sign, and $p_0=v^3>0$ implies
$v>0$.
\end{proof}

\section{Counting generic points with bounded Hessian coefficients}\label{sec:window}

Throughout this section $I\Subset\R$ is fixed.  Suppose that
$p_1/p_0\in I$, $Q\le p_0<2Q$, and
$$
\max_i|\Delta_i(\mathbf p)|\le M.
$$
If $M\ll Q^2$, then
$$
\Ht(F_{\mathbf p})\asymp_I Q.
$$
Indeed, since $I$ is compact, we have $|p_1|\ll_I Q$.  Moreover,
$$
|p_2|
=
\left|\frac{p_1^2+\Delta_0}{p_0}\right|
\ll_I Q
$$
and
$$
|p_3|
=
\left|\frac{p_1p_2+\Delta_1}{p_0}\right|
\ll_I Q.
$$
Since also $Q\le p_0<2Q$, it follows that
$$
\Ht(F_{\mathbf p})\asymp_I Q.
$$
Thus the dyadic condition $p_0\asymp Q$ is equivalent, up to constants
depending on $I$, to a height bound for the associated binary cubic.

%\subsection{Good approximation implies bounded Hessian coefficients}

We establish an elementary connection between Hessians and rational points near the Veronese curve.
% Simultaneous approximation with error $Q^{-\lambda}$ forces all three Hankel
% minors, equivalently all three Hessian coefficients, to be bounded in absolute
% value by a constant multiple of $Q^{1-\lambda}$.

\begin{lemma}\label{lem:approx-window}
Let $I\Subset\mathbb R$ be compact.  Suppose that $Q\le p_0<2Q$, $x\in I$, and
$|p_0x^i-p_i|\le E$ for $1\le i\le3$.  Then
$$
        |\Delta_i(\mathbf p)|\le C_I QE+C_IE^2\qquad(0\le i\le2).
$$
In particular, if $Q\ge1,\,E=Q^{-\lambda}$ and $\lambda\ge0$, then
$|\Delta_i(\mathbf p)|\ll_I Q^{1-\lambda}$.
\end{lemma}

\begin{proof}
Write $e_i=p_0x^i-p_i$, so that $p_i=p_0x^i-e_i$ and $|e_i|\le E$.  A direct
substitution gives
$$
\begin{aligned}
\Delta_0&=-p_0e_2+2p_0xe_1-e_1^2,\qquad
\Delta_1=-p_0e_3+p_0xe_2+p_0x^2e_1-e_1e_2,\\
\Delta_2&=-p_0xe_3-p_0x^3e_1+2p_0x^2e_2+e_1e_3-e_2^2.
\end{aligned}
$$
Since $x$ stays in the fixed compact interval $I$ and $p_0<2Q$, all three expressions
are $O_I(QE+E^2)$.  If $E=Q^{-\lambda}$, then $E^2\le QE$ for $Q\ge1$ and
$\lambda\ge0$.
\end{proof}

\subsection{Counting generic points}\label{sec:counting}

We use Davenport's theorem \cite{davenport_1951} on equivalence classes of irreducible binary cubic forms together with Hooley's additional calculation \cite[p.~153]{hooley_1968} for the number of equivalence classes of reducible binary cubic forms.  The binomially normalised lattice is a $\GL_2(\mathbb Z)$-invariant finite-index sublattice of the usual coefficient lattice, so restricting the counting result to this lattice can only decrease the number of classes. We combine Davenport's and Hooley's results in the following statement.

\begin{theorem}\label{thm:davenport}
For $D\ge 1$, the number of $\GL_2(\Z)$-equivalence classes of integral binary cubic forms $F$ with
$$
        0<|\Disc(F)|\le D
$$
is $O(D)$.
\end{theorem}

Our next goal is to count the number of cubic forms in each fixed class. Lemma~\ref{lem:sector-cover-count} below is a primitive lattice-point bound for sector covers. Lemma~\ref{lem:model-sector-cover} then constructs sector covers for the two model cubic-Hessian pairs. Later, we will convert the problem of counting cubic forms in each equivalence class to one of these model forms. Here by a sector we mean a set
$$
        \Sigma(R,\Theta)=\{r(\cos\theta,\sin\theta):R\le r<2R,
        \ \theta\in\Theta\},
$$
where $R>0$ and $\Theta\subset\R/2\pi\Z$ is an interval.  Its area is
$\asymp R^2|\Theta|$.

\begin{lemma}\label{lem:sector-cover-count}
Let $\Lambda\subset\R^2$ be unimodular, and let $\Lambda_{\rm prim}$ denote
the set of primitive vectors, i.e., the set of vectors $v\in \Lambda$ for which $v=kw$, with $w \in \Lambda$ and $k\in\ZZ$, implies $|k|=1$. If a set $E\subset\R^2$ is
covered by finitely many sectors $\Sigma_j$, then
$$
        \#(\Lambda_{\rm prim}\cap E)
        \ll
        \#\{j\}+\sum_j \operatorname{area}(\Sigma_j).
$$
The implied constant is absolute.
\end{lemma}

\begin{proof}
It is enough to prove the claim for one sector
$\Sigma=\Sigma(R,\Theta)$.  By subdividing $\Theta$ into at most four intervals,
we may assume that its angular length is at most $\pi/2$.
Then, in particular, two distinct directions occurring in $\Sigma$ cannot differ by
$\pi$. Write $\Lambda=A\Z^2$ with $|\det A|=1$. If
$z_i=Am_i\in\Lambda$, $i=1,2$, have distinct directions in $\Sigma$, then
they are non-collinear and
$$
        |\det(z_1,z_2)|=|\det A|\,|\det(m_1,m_2)|
$$
 is a positive integer, and therefore is at least one.  On the other hand, for
$z_1,z_2\in\Sigma$,
$$
        |\det(z_1,z_2)|\le 4R^2\,\angle(z_1,z_2),
$$
where the angle is the difference of their arguments, which lies in
$[0,\pi/2]$. Thus distinct directions of primitive lattice vectors in
$\Sigma$ are separated by $\gg R^{-2}$. The interval $\Theta$ therefore
contains $O(1+R^2|\Theta|)$ such directions. Each ray from the origin contains
at most one primitive vector of $\Lambda$. Hence
$$
        \#(\Lambda_{\rm prim}\cap\Sigma)\ll 1+R^2|\Theta|
        \ll 1+\operatorname{area}(\Sigma),
$$
and the lemma follows.
\end{proof}

We next verify that the specific model regions arising from binary cubics admit sector covers with the required total-area bound.

\begin{lemma}\label{lem:model-sector-cover}
Let $U\ge 1$, $T\ge 0$, and put
$$F_+(X,Y)=XY(X-Y),\qquad H_+(X,Y)=-\frac{1}{9}(X^2-XY+Y^2),$$
$$F_-(X,Y)=X^3+Y^3,\qquad H_-(X,Y)=XY.$$
Here, $H_\pm$ are the normalised Hessians of $F_\pm$. For each sign $\pm$, the region
$$\Omega_\pm(U,T)=\{(X,Y):U\le |F_\pm(X,Y)|<2U,\ |H_\pm(X,Y)|\le T\}$$
can be covered by sectors $\Sigma_j$ satisfying
\begin{equation}\label{eq:model-sector-cover}
\#\{j\}+\sum_j\operatorname{area}(\Sigma_j)\ll 1+T.
\end{equation}
The implied constant is absolute.
\end{lemma}

\begin{proof}
Write $\|(X,Y)\|=(X^2+Y^2)^{1/2}$.

\medskip
\noindent
\emph{The plus model.}
The inequalities
$$\frac{1}{2}(X^2+Y^2)\le X^2-XY+Y^2\le \frac{3}{2}(X^2+Y^2)$$
follow from
$$X^2-XY+Y^2=\frac{1}{2}(X^2+Y^2)+\frac{1}{2}(X-Y)^2$$
and $|XY|\le (X^2+Y^2)/2$. Hence $|H_+(X,Y)|\le T$ implies $\|(X,Y)\|\ll T^{1/2}$.

Also, $|X|,|Y|\le \|(X,Y)\|$ and $|X-Y|\le 2\|(X,Y)\|$, so
$$|F_+(X,Y)|=|XY(X-Y)|\le 2\|(X,Y)\|^3.$$
Thus $|F_+(X,Y)|\ge U$ implies $\|(X,Y)\|\gg U^{1/3}$. Therefore, for some absolute constants $c,C>0$,
$$\Omega_+(U,T)\subset\{(X,Y):cU^{1/3}\le \|(X,Y)\|\le CT^{1/2}\}.$$

If $CT^{1/2}<cU^{1/3}$, then $\Omega_+(U,T)$ is empty. Otherwise choose
$R_0$ with $R_0\le cU^{1/3}<2R_0$, put $R_j=2^jR_0$, and let $J$ be the
least non-negative integer such that $2R_J>CT^{1/2}$. Then the annuli
$$R_j\le \|(X,Y)\|<2R_j\qquad (0\le j\le J)$$
cover $\Omega_+(U,T)$, and $R_J\ll T^{1/2}$. Cover each annulus by four sectors of angular width $\pi/2$. The number of sectors is $O(1+\log(2+T))$, while their total area is at most the area of a disk of radius $2R_J$,
i.e., it is $\ll R_J^2\ll T$.
% $$\ll\sum_{j=0}^J R_j^2=R_J^2\sum_{m=0}^J4^{-m}\ll R_J^2\ll T.$$
Since $1+\log(2+T)\ll 1+T$, this proves \eqref{eq:model-sector-cover} for $\Omega_+(U,T)$.

\medskip
\noindent
\emph{The minus model.}
The reflection $(X,Y)\mapsto(Y,X)$ preserves $\Omega_-(U,T)$, Euclidean norm, area, and sectors. It is therefore enough to cover the part of $\Omega_-(U,T)$ where $|X|\ge |Y|$.

For each integer $k$, define
$$
\mathcal S_k=\{(X,Y)\in\Omega_-(U,T):|X|\ge |Y|,\ 2^k\le |X|<2^{k+1}\}.
$$
These sets form a dyadic decomposition of the part of $\Omega_-(U,T)$ where $|X|\ge |Y|$. Fix $k$ and write $L=2^k$, so that $L\le |X|<2L$ on $\mathcal S_k$.

For every $(X,Y)\in\mathcal S_k$,
$$
L\le \|(X,Y)\|\le \sqrt{2}|X|<2\sqrt{2}L<4L.
$$
Hence
$$
\mathcal S_k\subset\{(X,Y)\in \mathcal S_k:L\le \|(X,Y)\|<2L\}\cup\{(X,Y)\in\mathcal S_k:2L\le \|(X,Y)\|<4L\}.
$$
Thus, for a fixed $k$, only these two annuli need to be considered. We split
$\mathcal S_k$ according to whether $|X+Y|/|X|$ is at least $1/2$ or less than
$1/2$. These will be called the non-cancellation and cancellation cases,
respectively.

\medskip
\noindent
\emph{Non-cancellation.}
Suppose that $|X+Y|\ge |X|/2$, and put $s=Y/X$. Since $|Y|\le |X|$, we have $-1\le s\le 1$. Moreover,
$$
|X+Y|=|X|\,|1+s|.
$$
Hence the non-cancellation condition implies $1+s\ge 1/2$, or equivalently $s\ge -1/2$. Thus $s\in[-1/2,1]$, and on this interval
$$
X^3+Y^3=X^3(1+s^3),\qquad \frac{7}{8}\le 1+s^3\le 2.
$$
It follows that $|X^3+Y^3|\asymp |X|^3$. Since $L\le |X|<2L$ on $\mathcal S_k$, we obtain
$$
|X^3+Y^3|\asymp L^3.
$$
Points of $\Omega_-(U,T)$ also satisfy $U\le |X^3+Y^3|<2U$. Therefore, if the non-cancellation part of $\mathcal S_k$ is non-empty, then
$$
L^3\asymp U,
$$
or equivalently $L\asymp U^{1/3}$. Since $L=2^k$, this restricts $k$ to an interval of bounded length. Hence only $O(1)$ integers $k$ can contribute to the non-cancellation part.

The Hessian condition gives $X^2|s|=|XY|\le T$. Since $|X|\ge L$, every admissible slope lies in
$$
J_L=[-1/2,1]\cap[-T/L^2,T/L^2],\qquad |J_L|\ll\min\{1,T/L^2\}.
$$
For $X>0$, the direction is $\arctan s$; for $X<0$, it is the direction
obtained by adding $\pi$. Since $s\mapsto\arctan s$ is Lipschitz on
$[-1,1]$, the admissible directions lie in at most two angular intervals of
total length $O(\min\{1,T/L^2\})$. Intersecting them with the two annuli
above produces at most four sectors, with total area
$$
\ll L^2\min\{1,T/L^2\}\ll T.
$$
Since only $O(1)$ integers $k$ contribute, the whole non-cancellation part is covered by $O(1)$ sectors of total area $O(T)$.

\medskip
\noindent
\emph{Cancellation.}
Suppose that $|X+Y|<|X|/2$, and put $t=(X+Y)/X=1+Y/X$. Since $|Y|\le |X|$, we have $0\le t<1/2$. As $Y=X(t-1)$,
$$
|XY|=X^2(1-t)\asymp L^2,
$$
so the cancellation part of $\mathcal S_k$ can be non-empty only if $L^2\ll T$.

Furthermore,
$$
X^3+Y^3=X^3h(t),\qquad h(t)=(t-1)^3 + 1 = 3t-3t^2+t^3.
$$
Notice that $h$ is strictly increasing for $t>0$. Since $L\le |X|<2L$, every admissible value of $t$ lies in the interval
$$
K_L=\left\{0\le t<\frac{1}{2}:\frac{U}{8L^3}\le h(t)<\frac{2U}{L^3}\right\}.
$$
The image $h(K_L)$ has length $O(U/L^3)$, and $h'(t)\gg 1$ on $[0,1/2]$. Hence the mean value theorem gives
$$
|K_L|\ll U/L^3.
$$
For $X>0$, the direction is $\arctan(t-1)$; for $X<0$, it is the direction
obtained by adding $\pi$. Since $t\mapsto\arctan(t-1)$ is Lipschitz on
$[0,1/2]$, the admissible directions lie in at most two angular intervals of
total length $O(U/L^3)$. Intersecting them with the two annuli above produces
at most four sectors, with total area
$$
\ll L^2\frac{U}{L^3}=\frac{U}{L}.
$$

If the cancellation part of $\mathcal S_k$ is non-empty, then also $U\ll L^3$, because $0\le h(t)\le h(1/2)$. Together with $L^2\ll T$, this gives
$$
U^{1/3}\ll L\ll T^{1/2}.
$$
Since $L=2^k$, the contributing integers $k$ satisfy
$$
U^{1/3}\ll 2^k\ll T^{1/2}.
$$
Summing the sector areas over these integers gives
$$
\sum_k\frac{U}{2^k}\ll U^{2/3}.
$$
Indeed, this is a geometric series dominated by its first term, corresponding to $2^k\asymp U^{1/3}$. If the range is non-empty, then $U^{1/3}\ll T^{1/2}$, and hence $U^{2/3}\ll T$. Thus the total sector area in the cancellation part is $O(T)$.

The number of contributing integers $k$ is $O(1+\log(2+T))$, and each set $\mathcal S_k$ requires only $O(1)$ sectors. Hence the total number of sectors is $O(1+T)$.

Combining the two parts and applying the reflection $(X,Y)\mapsto(Y,X)$ proves \eqref{eq:model-sector-cover} for $\Omega_-(U,T)$.
\end{proof}

The next lemma counts primitive lattice points lying in the region where the cubic $F$ has size about $Q$ and its Hessian has size at most $M$. We use the convention that if $\vu=(u_1,u_2)\in\Lambda$, the expressions $F(\vu)$ and $H(\vu)$ below mean $F(u_1,u_2)$ and $H(u_1,u_2)$, respectively.

\begin{lemma}\label{lem:arch-slice}
Let
$$
        F(X,Y)=aX^3+3bX^2Y+3cXY^2+dY^3
$$
be a real binary cubic with nonzero discriminant $\Delta=\Disc(F)$, and let
$H=\Hess(F)$.  Assume $|\Delta|\ge1$.  Let $\Lambda\subset\R^2$ be a
unimodular lattice.  For
$Q\ge\max\{2,|\Delta|^{1/4}\}$ and $M\ge1$, let
$N_\Lambda(F;Q,M)$ be the number of primitive vectors $\vu\in\Lambda$ such
that
$$
        Q\le |F(\vu)|<2Q,
        \qquad |H(\vu)|\le M.
$$
Then
\begin{equation}\label{eq:arch-slice}
        N_\Lambda(F;Q,M)
        \ll 1+\frac{M}{|\Delta|^{1/2}}.
\end{equation}
The implied constant is absolute.
\end{lemma}

\begin{proof}
We reduce $F$ to one of the two model cubics from Lemma~\ref{lem:model-sector-cover}. Their discriminants are
$$
\Disc(F_+)=1,\qquad \Disc(F_-)=-27.
$$
If $\Delta>0$, the three projective roots of $F$ are distinct and real, so
there exists a real projective transformation $g$ that sends them, in a chosen order, to the
three roots of $F_+$. If $\Delta<0$, $F$ has one real projective root and a
non-real conjugate pair of roots. In this case choose the projective transformation $g$ that sends the real root of $F$ to the real root of $F_-$ and
the conjugate pair to the conjugate pair of $F_-$ in the corresponding
order. Observe that $g$ commutes with complex conjugation and is therefore represented by a real matrix.

Choose a real matrix representing $g$, and rescale it so that $|\det g|=1$. We will also call this matrix $g$. Since $g$ sends the roots of
$F$ to those of the appropriate model form $F_\sigma$, $\sigma \in\{-,+\}$, both $F$ and $F_\sigma\circ g$ have the
same projective roots, with the same multiplicities, and therefore differ by
a nonzero scalar which we call $\kappa\in\RR\setminus\{0\}$. We therefore obtain the equation
\begin{equation}\label{eq:model-reduction}
        F=\kappa(F_\sigma\circ g).
\end{equation}
By well-known formulae that follow, for example, from~\eqref{eq:cubic-hess-disc} and~\eqref{eq:hessian-covariance}, one gets
$$
\Delta= \Disc(F) = \kappa^4\Disc(F_\sigma\circ g) = \kappa^4 (\det g)^6 \Disc(F_\sigma) = \kappa^4\Disc(F_\sigma),
$$
which immediately implies
\begin{equation}\label{eq:kappa-size}
        |\kappa|=
        \left(\frac{|\Delta|}{|\Disc(F_\sigma)|}\right)^{1/4}.
\end{equation}
Moreover, by Hessian covariance and $|\det g|=1$,
\begin{equation}\label{eq:model-hessian-reduction}
        H=\kappa^2(H_\sigma\circ g).
\end{equation}

The lattice $g\Lambda$ is unimodular, and $g$ preserves primitivity.  Hence,
 the change of variables $\vv=g\vu$ leads to $N_\Lambda(F;Q,M) = N_{g\Lambda}(F_\sigma;U,T)$ with parameters
$$
        U=\frac{Q}{|\kappa|},\qquad
        T=\frac{M}{\kappa^2}.
$$
Indeed, \eqref{eq:kappa-size} gives
$$
U=Q\frac{|\Disc(F_\sigma)|^{1/4}}{|\Delta|^{1/4}}\ge1,
$$
because $|\Disc(F_\sigma)|\in\{1,27\}$ and $Q\ge|\Delta|^{1/4}$.
Lemma~\ref{lem:model-sector-cover}, followed by Lemma~\ref{lem:sector-cover-count}, therefore gives
$$
        N_\Lambda(F;Q,M)\ll 1+T.
$$
Finally, $|\Disc(F_\sigma)|\in\{1,27\}$, so
$T\ll M|\Delta|^{-1/2}$.  This proves \eqref{eq:arch-slice}.
\end{proof}

The next lemma counts the forms satisfying the coefficient restrictions within a
fixed equivalence class.

\begin{lemma}\label{lem:class-slice}
Let $\mathcal C$ be a $\GL_2(\Z)$-equivalence class of primitive integral
binomially normalised binary cubic forms with nonzero discriminant $d$.
Let $I\Subset\R$ be a compact interval, let $M\ge1$, and assume
$Q\ge\max\{2,|d|^{1/4}\}$.  The number of forms
$$
        F(X,Y)=q_0X^3+3q_1X^2Y+3q_2XY^2+q_3Y^3
$$
in $\mathcal C$ satisfying
$$
        Q\le q_0<2Q,
        \qquad q_1/q_0\in I,
        \qquad \max_i|\Delta_i(\mathbf q)|\le M
$$
is
\begin{equation}\label{eq:class-slice}
        \ll_I 1+\frac{M}{|d|^{1/2}}.
\end{equation}
\end{lemma}

\begin{proof}
Every form in the class is produced by at least one matrix in
$\GL_2(\Z)$. Different matrices may produce the same form because of the
stabiliser, so counting matrices gives an upper bound. Since
$\SL_2(\Z)$ has index two in $\GL_2(\Z)$, a $\GL_2(\Z)$-class is the
union of at most two  $\SL_2(\Z)$-classes. It therefore
suffices to treat one $\SL_2(\Z)$-class. Choose a representative $F_0$ and write
$H_0=\Hess(F_0)$. Every form in the class is of the form
$$
        F=F_0(\alpha X+\beta Y,\gamma X+\delta Y),
        \qquad
        g=\begin{pmatrix}\alpha&\beta\\ \gamma&\delta\end{pmatrix}
        \in\SL_2(\Z).
$$
Let $\vu=(\alpha,\gamma)$ be the first column of $g$. Setting $Y=0$ and
using the homogeneity of $F_0$ gives $q_0=F_0(\vu)$. Moreover, by
\eqref{eq:hessian-covariance} and $\det g=1$, the coefficient of $X^2$ in
$\Hess(F)$ is $H_0(\vu)$. Thus
$$
q_0=F_0(\vu),\qquad \Delta_0=H_0(\vu).
$$
Consequently, the stated conditions imply
$$
        Q\le |F_0(\vu)|<2Q,
        \qquad |H_0(\vu)|\le M.
$$
The vector $\vu$ is primitive, so Lemma~\ref{lem:arch-slice} gives
$O(1+M/|d|^{1/2})$ possibilities for $\vu$.

Fix such a first column and choose one $\vv_0=(b_0,d_0)\in\Z^2$ with
$\det(\vu,\vv_0)=1$.  Every other possible second column is
$\vv=\vv_0+n\vu$ with $n\in\Z$.  Replacing $\vv_0$ by $\vv_0+n\vu$ changes the form to
$$
        F_0(\vu(X+nY)+\vv_0Y).
$$
Thus $q_0$ is unchanged, while $q_1$ is replaced by $q_1+nq_0$.
Consequently $q_1/q_0$ is replaced by $q_1/q_0+n$. Since this ratio must
lie in the fixed compact interval $I$, only $O_I(1)$ values of $n$ are
possible. This proves \eqref{eq:class-slice}.
\end{proof}

We now sum the fixed-class estimate over the irreducible and reducible classes, using Theorem~\ref{thm:davenport}.

\begin{proposition}\label{prop:nondeg-count}
Let $I\Subset\R$ and $C_0\ge1$.  Uniformly for $Q\ge2$ and
$1\le M\le C_0Q$, the number of primitive $\mathbf q\in\Z^4$ such
that
$$
Q\le q_0<2Q,
\qquad q_1/q_0\in I,
\qquad \max_i|\Delta_i(\mathbf q)|\le M,
\qquad \Delta_1^2-4\Delta_0\Delta_2\neq0
$$
is $O_{I,C_0}(M^2)$.
\end{proposition}

\begin{proof}
Let $F=F_{\mathbf q}$. Since the Hessian coefficients satisfy
$$
\max_{0\le i\le2}|\Delta_i(\mathbf q)|\le M
$$
and
$$
\Delta_1^2-4\Delta_0\Delta_2\ne0,
$$
the discriminant of the Hessian is a nonzero integer of absolute value
$O(M^2)$. Hence, by~\eqref{eq:cubic-hess-disc},
$$
0<|\Disc(F)|\ll M^2.
$$

Theorem~\ref{thm:davenport} shows that
the number of relevant irreducible and reducible classes with
$D<|\Disc(F)|\le2D$ is $O(D)$.

For a class of discriminant $d$ in this dyadic block, Lemma~\ref{lem:class-slice} limits its contribution by
$$
        \ll_I 1+\frac{M}{|d|^{1/2}}
        \ll 1+\frac{M}{D^{1/2}}.
$$
The hypothesis $Q\ge |d|^{1/4}$ required in that lemma holds once $Q$
exceeds a constant depending only on $C_0$, because
$|d|\ll M^2\le C_0^2Q^2$. If $Q$ is below that constant, then
$Q$ and $M$ are bounded in terms of $C_0$, and in view of $|q_1|, |q_2|, |q_3|\ll_I q_0\ll Q$, only $O_{I,C_0}(1)$ such forms occur. Hence, the dyadic
block contributes
$$
        \ll_I D+MD^{1/2}.
$$
Summing over dyadic $D\ll M^2$ gives $O_{I,C_0}(M^2)$, as required.
\end{proof}

\subsection{Proof of Theorem~\ref{th2}}

We first formulate the estimate that follows from the previous section.  For $Q>0$ and $0<\delta<1/2$, let
\begin{align*}
P_g^{\rm prim}(Q,\delta):=
\{\mathbf p\in\ZZ^4_{\rm prim}:\, Q\le p_0<2Q,\,
\mathbf p\text{ is generic},\inf_{x\in I}\max_{1\le i\le3}|p_0x^i-p_i|\le\delta\},
\end{align*}
and put $N_g^{\rm prim}(Q,\delta):=\#P_g^{\rm prim}(Q,\delta)$.

\begin{lemma}\label{lem:primitive-generic-count}
Uniformly for $Q>0$ and $0<\delta<1/2$,
\begin{equation}\label{eq:primitive-generic-count}
N_g^{\rm prim}(Q,\delta)\ll_I (1+Q)^2\delta^2.
\end{equation}
In particular, for $Q\ge1$ the right-hand side is $O_I(Q^2\delta^2)$.
\end{lemma}

\begin{proof}
Assume first that $Q\ge2$, and choose a compact interval $I^+$ whose
interior contains the closed $1$-neighbourhood of $I$.  Let
$\mathbf p\in P_g^{\rm prim}(Q,\delta)$, and choose $x\in I$ such that
$$
\max_{1\le i\le3}|p_0x^i-p_i|\le\delta.
$$
The first approximation inequality gives
$$
\left|\frac{p_1}{p_0}-x\right|\le\frac{\delta}{Q},
$$
so $p_1/p_0\in I^+$.  By Lemma~\ref{lem:approx-window},
$$
\max_{0\le i\le2}|\Delta_i(\mathbf p)|
\le C_{I^+}(Q\delta+\delta^2)=:A.
$$
As $\mathbf p$ is generic, Proposition~\ref{prop:special-hessian-correspondence} implies
$$
\Delta_1^2-4\Delta_0\Delta_2\ne0.
$$
Since the $\Delta_i(\mathbf p)$ are integers, the vector $\vp$ can exist only
if $A\ge1$.  We may therefore put $M=\lceil A\rceil$, in which case
$$
M\le2A\ll_I Q\delta\ll Q.
$$
Proposition~\ref{prop:nondeg-count} now gives
$$
N_g^{\rm prim}(Q,\delta)\ll_I M^2\ll_I Q^2\delta^2.
$$

It remains to consider $Q<2$.  Then $p_0$ can take only finitely many
values.  Since $x$ lies in a fixed compact interval and $\delta<1/2$, only
finitely many integer vectors $\mathbf p$ can occur.  None of the generic
ones lies exactly on the Veronese curve, and hence the distance of this finite
set from the curve has a positive minimum. Thus the set is empty for
sufficiently small $\delta$, while for the remaining $\delta$ the claimed bound
follows after enlarging the implied constant, if needed.
\end{proof}

We next remove the primitivity restriction.  Let $N_g^{\rm dyad}(Q,\delta)$
denote the same count as $N_g^{\rm prim}(Q,\delta)$, but with
$\mathbf p\in\ZZ^4$ arbitrary rather than primitive.

\begin{lemma}\label{lem:remove-primitivity-count}
Uniformly for $Q>0$ and $0<\delta<1/2$,
\begin{equation}\label{eq:all-generic-dyadic}
N_g^{\rm dyad}(Q,\delta)\ll_I (1+Q)^2\delta^2.
\end{equation}
In particular, for $Q\ge1$ the right-hand side is $O_I(Q^2\delta^2)$.
\end{lemma}

\begin{proof}
Write every counted vector uniquely as $\mathbf p=d\mathbf q$, where
$d=\gcd(p_0,p_1,p_2,p_3)$ and $\mathbf q$ is primitive.  Multiplication by
$d$ multiplies the associated cubic by $d$ and its discriminant by $d^4$,
hence the property of a point being generic is preserved.  Moreover,
$$
Q\le p_0< 2Q,
\qquad
\inf_{x\in I}\max_i|p_0x^i-p_i|\le\delta
$$
imply
$$
\frac Qd\le q_0<\frac{2Q}{d},
\qquad
\inf_{x\in I}\max_i|q_0x^i-q_i|\le\frac\delta d.
$$
Consequently, Lemma~\ref{lem:primitive-generic-count} gives
\begin{align*}
N_g^{\rm dyad}(Q,\delta)
&\le\sum_{1\le d<2Q}
N_g^{\rm prim}\left(\frac Qd,\frac\delta d\right)\\
&\ll_I\sum_{1\le d<2Q}
\left(1+\frac Qd\right)^2\frac{\delta^2}{d^2}\\
&\ll_I\delta^2\sum_{d\ge1}\left(\frac1{d^2}+\frac{Q^2}{d^4}\right)
\ll_I (1+Q)^2\delta^2.
\end{align*}
Thus, passing from primitive to arbitrary generic points only changes the
implied constant in the upper bound.
\end{proof}

Finally, we deduce Theorem~\ref{th2}.

\begin{proof}
Cover the denominator range $1\le p_0\le Q$ with dyadic intervals
$$
2^{-j}Q\le p_0<2^{-j+1}Q.
\qquad (j\ge0),
$$
stopping when the lower endpoint is below $1$.  Applying
Lemma~\ref{lem:remove-primitivity-count} to each block gives
$$
N_{g,\VVV_3}(Q,\delta)\ll_I\delta^2\sum_j\bigl(1+2^{-j}Q\bigr)^2\ll_I Q^2\delta^2.
$$
This proves~\eqref{th2_eq}.
\end{proof}

\section{The upper bound for generic points}\label{sec:upper}

% We first state the reduction to primitive points in a form that preserves
% the property of being generic or special.

% \begin{lemma}\label{lem:primitive-reduction}
% Let $\lambda\ge0$, and suppose that an irrational number $x$ admits
% infinitely many approximating vectors from one fixed class, either generic or
% special.  Then it admits infinitely many primitive approximating vectors from
% the same class, with unbounded first coordinate.
% \end{lemma}

% \begin{proof}
% For every approximating vector in the chosen class, divide by
% $d=\gcd(q_0,q_1,q_2,q_3)$ and write $\mathbf q'=d^{-1}\mathbf q$. Then
% $$
% |q_0'x^i-q_i'|
% =d^{-1}|q_0x^i-q_i|
% <d^{-1-\lambda}(q_0')^{-\lambda}
% \le(q_0')^{-\lambda}.
% $$
% The associated cubic is multiplied by $d^{-1}$, so its discriminant is
% multiplied by $d^{-4}$.  Thus both the nonzero-discriminant class and the
% zero-discriminant class are preserved.  If the first
% coordinates $q_0'$ were bounded, only finitely many primitive vectors could
% occur. One of them would arise with multipliers $d\to\infty$, forcing
% $q_0'x=q_1'$ and hence $x\in\QQ$, a contradiction.
% \end{proof}

We now prove Corollary~\ref{cor:generic-upper}.

\begin{proof}
It is enough to work on a
compact interval $I\Subset\RR$.  For a dyadic number $Q=2^k$, let
$\mathcal P_Q$ be the set of generic integer vectors satisfying
\begin{equation}\label{eq2}
Q\le p_0<2Q,
\qquad
\inf_{x\in I}\max_{1\le i\le3}|p_0x^i-p_i|\le Q^{-\lambda}.
\end{equation}
For all sufficiently large $Q$, Lemma~\ref{lem:remove-primitivity-count} gives
\begin{equation}\label{eq:generic-dyadic-count}
\#\mathcal P_Q
\ll_I Q^{2-2\lambda}.
\end{equation}
One can easily check that for each $\mathbf p\in\mathcal P_Q$, the set of $x\in I$ satisfying $$\max_{1\le i\le3}|p_0x^i-p_i|<p_0^{-\lambda}$$ is contained in an interval of length $O(Q^{-1-\lambda})$.  Hence the sum of the $s$-powers of all covering
intervals is bounded by
$$
\sum_{Q=2^k}Q^{2-2\lambda-s(1+\lambda)}.
$$
If $\frac13\le\lambda\le 1$, this converges whenever $s>\frac{2-2\lambda}{1+\lambda}$. Then the Hausdorff-Cantelli lemma implies $\dim S_g(I,\lambda) \le \frac{2-2\lambda}{1+\lambda}$. If $\lambda> 1$, the condition~\eqref{eq:generic-dyadic-count} implies that the set $\PPP_Q$ is empty for large enough $Q$.  Therefore
$$
S_g(I,\lambda) = \emptyset \quad\mbox{and}\quad \dim S_g(I,\lambda) = 0
$$
as claimed.
\end{proof}

\section{Special points: lower bound for the Hausdorff dimension}\label{sec:special-lower}

By $\PPP_s(\lambda)$ we denote the set of all special
$\vp\in\ZZ_{>0}\times\ZZ^3$ such that there exists $\xi\in\RR$
for which~\eqref{main_ineq} is satisfied. Define
$\RRR(\vp)=\RRR(\vp,\lambda)$ to be the set of $\xi\in\RR$
that satisfy~\eqref{main_ineq} for a given $\vp$.
%$$
%Q_s(
%$$
%
%By $Q(\lambda) = Q_n(\lambda)$ we denote the set of all $\vq\in
%\ZZ^{n+1}$ such that there exists $x\in \RR$ such that
%\begin{equation}\label{def_qlambda}
%\max_{1\le i\le n} |q_0x^i - q_i| < ||\vq||_\infty^{-\lambda}.
%\end{equation}
%Sometimes it is convenient to write $\vq$ as a pair $(q_0,
%\vq^+)\in\ZZ\times\ZZ^n$. Then~\eqref{def_qlambda} can be rewritten
%as $||q_0\vv(x) - \vq^+||_\infty \ll ||\vq||^{-\lambda}$. By
%$R(\vq,\lambda) = R_n(\vq,\lambda)$ we denote the set of $x\in\QQ$
%that satisfy~\eqref{def_qlambda} for a given $\vq\in Q(\lambda)$.
%
%
%
%
%\begin{proposition}
%Let
%$$
%S^*(\lambda):= \{x\in\RR: x\in R(\vq, \lambda)\;\mbox{ for i.m. }\;
%\vq\in Q^*(\lambda)\}.
%$$
%Then for $\lambda\ge 3/5$, $\dim W^*(\lambda) \ge
%\frac{2-2\lambda}{1+\lambda}$.
%\end{proposition}
%
%Notice that $S^*(\lambda)\subset S(\lambda)$ therefore together with
%~\cite{beresnevich_2012}, the proposition immediately implies that
%$\dim S(\lambda)\ge \frac{2-2\lambda}{1+\lambda}$ for all
%$\lambda\ge 1/3$.
%Yes. The lower bound comes from building a Cantor subset using only
%one carefully chosen dyadic range of (r_k)�s.
Consider the following points which lie on $L_{P_{u,v}}$ and hence
are special:
\begin{equation}\label{def_p12}
\vq_1:= (v^3, uv^2, u^2v, u^3), \qquad \vq_2:= \left(v^2r_0,
v(ur_0+1), u(ur_0+2), \frac{u^2(ur_0+3)}{v}\right),
\end{equation}
where $0\le r_0<v$ is such that $v\mid ur_0+3$. For $k\in \ZZ$
define $\vp(k):= k\vq_1 + \vq_2$ and $r_k:= r_0 + kv$. Clearly, all
the points $\vp(k)$ are special. For the lower bound of the Hausdorff dimension, we only consider reduced
fractions $u/v$ such that $0\le u/v\le1$ and $v>0$.

From~\eqref{def_la}, direct computations show that the height of $L_{P_{u,v}}$ equals
$$
H(L_{P_{u,v}}) = \big\|(u^2,-2uv,v^2,0)\wedge(0,u^2,-2uv,v^2)\big\|_2.
$$
One can verify that $H(L_{P_{u,v}})=\|\vq_1\wedge\vq_2\|_2$, and therefore the lattice $L_{P_{u,v}}\cap\ZZ^4$ is generated by $\vq_1$ and $\vq_2$.
\begin{lemma}\label{lem1}
Assume $0\le u/v\le1$ and let $\lambda<1$. There exists a constant $c>0$ such that if $k\ge
cv^{\frac{3\lambda-1}{1-\lambda}}$ then $\vp(k)\in \PPP_s(\lambda)$.
\end{lemma}

\begin{proof} We write $\vp(k) = (p_0, p_1, p_2, p_3)$.

Simple calculations show that
\begin{equation}\label{eq3}
\left|p_0 \cdot\left(\frac{p_1}{p_0}\right)^2 - p_2\right| =
\frac{1}{|r_k|}; \qquad\left|p_0\cdot\left(\frac{p_1}{p_0}\right)^3
- p_3\right| = \frac{|3ur_k+1|}{vr_k^2}\ll \frac{1}{|r_k|}.
\end{equation}

Since $0\le r_0<v$ and $k\ge1$, we have $r_k\asymp vk$ and
$$
\max\left\{\frac1{r_k},\frac{|3ur_k+1|}{vr_k^2}\right\}\ll\frac1{r_k}.
$$
Moreover, if
$k\ge cv^{\frac{3\lambda-1}{1-\lambda}}$, then
$$
\frac{r_k^{1-\lambda}}{v^{2\lambda}}\gg c^{1-\lambda}\quad \Longleftrightarrow \quad \frac{1}{r_k}\ll \frac{1}{c^{1-\lambda}(v^2r_k)^\lambda} = c^{\lambda-1} p_0^{-\lambda}.
$$
Thus, by choosing $c$ sufficiently large, both nonzero errors in~\eqref{eq3}
are less than $\frac14 p_0^{-\lambda}$. This proves
$\vp(k)\in \PPP_s(\lambda)$.
\end{proof}

In the following construction we assume that
$\frac35\le\lambda<1$. For the point $\vp(k)$ one has $x_0:=\frac{p_1}{p_0}=\frac{u}{v}+\frac{1}{vr_k}$.
By the proof of Lemma~\ref{lem1}, the errors at $x_0$ in the second
and third coordinates are at most $\frac14p_0^{-\lambda}$.
Since $x_0$ stays in a fixed bounded interval, the mean value theorem
shows that there exists $c_1>0$ such that
$$
\RRR(\vp(k))\supset B\left(\frac{u}{v}+\frac{1}{vr_k}, \frac{c_1}{(v^2r_k)^{1+\lambda}}\right).
$$
% In further arguments we assume that $\lambda<1$. Notice that for the point $\vp(k)$ one has $\frac{p_1}{p_0} =
% \frac{u}{v} + \frac{1}{vr_k}$. Hence there exists a constant $c_1>0$
% such that as soon as $k$ satisfies Lemma~\ref{lem1},
% $$
% \RRR(\vp(k)) \supset B\left(\frac{u}{v}+\frac{1}{vr_k},
% \frac{c_1}{(v^2r_k)^{1+\lambda}}\right).
% $$
Denote the interval on the right-hand side by $R(u,v,k)$. Now, the key
idea of the proof is to compute the lower bound for the Hausdorff
dimension of such $\xi\in\RR$ that fall inside infinitely many
intervals $R(u,v,k)$. Notice that these intervals form clusters
around each rational number $u/v$.

For convenience, define
$\alpha=\frac{3\lambda-1}{1-\lambda}$ and fix a constant
$K>c\,2^\alpha$. For fixed parameters $V>1$ and
$\gamma\ge\alpha$, consider the set of intervals
$$
\vR(\lambda,V,\gamma):=
\{R(u,v,k):0\le u/v\le1,\ \gcd(u,v)=1,\ V\le v<2V,\
KV^\gamma\le k<2KV^\gamma\}.
$$
By the choice of $\gamma$ and Lemma~\ref{lem1}, the union of intervals from $\vR(\lambda, V, \gamma)$ is a subset of $\bigcup_{\vp\in \PPP_s(\lambda)} \RRR(\vp)$. Also,
by convention we define $\vR(\lambda,1,\gamma):= \{[0,1]\}$. The
number of distinct pairs $(u,v)$ with $\gcd(u,v)=1$ in this block is
$\asymp V^2$. Moreover, for each fixed $(u,v)$ the number of distinct $k$ for $R(u,v,k)\in \vR(\lambda, V,\gamma)$ is $\asymp V^\gamma$. So the total number of intervals in $\vR(\lambda,V,\gamma)$ is
$\asymp V^{\gamma+2}$.

Each interval in $\vR(\lambda, V,\gamma)$ has radius
$$
\asymp (v^3k)^{-1-\lambda} \asymp V^{-(\gamma+3)(1+\lambda)} =:
\rho_V.
$$
Notice that intervals with the same pair $u,v$ form a cluster of
length $ (v^2k)^{-1} \asymp V^{-2-\gamma} =: L_V$.

%There are (\asymp V^2) such clusters.
%
Consider two consecutive intervals $R(u,v,k)$ and $R(u,v,k+1)$ from
$\vR(\lambda, V,\gamma)$. The distance between their centres is
$$
\frac{1}{vr_k} - \frac{1}{vr_{k+1}} = \frac{1}{r_kr_{k+1}} \gg
V^{-2-2\gamma}.
$$
Notice that for $\gamma< \frac{1+3\lambda}{1-\lambda}$ we have
$$
V^{-2-2\gamma} > V^{-(3+\gamma)(1+\lambda)} \asymp \rho_V,
$$
therefore these two intervals are disjoint. In the rest of the proof
we assume this upper bound on $\gamma$. Also, the distance between
the centres of two clusters is $\gg V^{-2}$, which is larger
than $L_V$. The conclusion is that all intervals in $\vR(\lambda,
V,\gamma)$ are disjoint.

Choose a rapidly increasing sequence $V_n$ such that $V_1=1$ and,
for each $n\ge 1$ and $R(u,v,k)\in \vR(\lambda,V_n,\gamma)$, the
number of rationals $u^+/v^+$ with $V_{n+1}\le v^+<2V_{n+1}$ in
the middle third of $R(u,v,k)$ is
$$
\asymp V_{n+1}^2|R(u,v,k)|.
$$
We choose $V_{n+1}$ sufficiently large so that the whole cluster
associated with each such $u^+/v^+$ is contained in $R(u,v,k)$.
Such a choice is possible since these rational numbers are
equidistributed and the corresponding clusters have length
$O(L_{V_{n+1}})\to0$. Thus each element of $\vR(\lambda,V_n,\gamma)$ contains $\asymp V_{n+1}^2|R(u,v,k)|$ clusters and $\asymp V_{n+1}^{2+\gamma}|R(u,v,k)|$ intervals from $\vR(\lambda,V_{n+1},\gamma)$. Notice that by
construction,
$$
\KKK(\lambda, \gamma):= \bigcap_{n=1}^\infty \III(\lambda, V_n,
\gamma) := \bigcap_{n=1}^\infty \bigcup_{I\in \vR^* (\lambda, V_n,
\gamma)} I
$$
is a subset of $S_s(\lambda)$. Here, $\III(\lambda,1,\gamma)$ is the
interval $[0,1]$ and $\vR^{*}(\lambda, V_{n+1}, \gamma)$ consists of
all $R(u,v,k)\in \vR(\lambda,V_{n+1}, \gamma)$ which are contained in $\III(\lambda,V_n,\gamma)$.

We construct a measure $\mu$ supported on $\KKK(\lambda,\gamma)$ and
then apply the mass distribution principle to get the lower bound for
$$
\dim\big(\limsup\limits_{n\to\infty} \vR(\lambda,V_n,\gamma)\big)\le
\dim S_s(\lambda).
$$
We set $\mu([0,1]):=1$. For each $n\ge1$ we equally distribute the measure of each retained level-$n$ interval among all $R(u^+,v^+,k^+)\in\vR(\lambda,V_{n+1},\gamma)$ contained in it. This gives
$$
\mu(R(u^+,v^+,k^+))\asymp \frac{\mu(R(u,v,k))}{V_{n+1}^{2+\gamma}
|R(u,v,k)|}.
$$
%The support of this measure is a Cantor set
%$$
%\KKK(\lambda, \gamma):= \bigcap_{n=1}^\infty \III(\lambda, V_n,
%\gamma) := \bigcap_{n=1}^\infty \bigcup_{I\in Q^{**} (\lambda, V_n,
%\gamma)} I,
%$$
%where $\III(\lambda,1,\gamma)$ is the interval $[0,1]$ and
%$Q^{**}(\lambda, V_{n+1}, \gamma)$ consists of all $R(u,v,k)\in
%Q^*(\lambda,V_{n+1}, \gamma)$ which belong to $\III(\lambda, V_n,
%\gamma)$.
Also, the length of each interval in $\vR(\lambda,V_n,\gamma)$ is
$\asymp \rho_{V_n} =: \rho_n$ and the length of each cluster at
level $n$ is $\asymp L_{V_n} =: L_n$. Recall,
$$
\rho_n = V_n^{-(\gamma+3)(1+\lambda)},\quad L_n = V_n^{-2-\gamma}.
$$
Fix
$$
0<s<s'<\min\left\{\frac{2}{2+\gamma}, \frac{\gamma+2}{(\gamma+3)(1+\lambda)}\right\}.
$$

Consider an arbitrary interval $I$ intersecting $\KKK(\lambda,\gamma)$
with $0<|I|<1$. Then there exists $n\in\NN$ such that
$$
\rho_{n+1}\le |I|<\rho_n.
$$
For such $n$, the interval $I$ intersects $O(1)$ intervals
$J\in\vR^*(\lambda,V_n,\gamma)$. Fix one such $J$ and estimate
$\mu(I\cap J)/\mu(J)$.

{\bf Case 1.} Let $|I|\ge V_{n+1}^{-2}$. In this case $I$ intersects
at most $5|I|V^2_{n+1}$ clusters from $\vR^*(\lambda, V_{n+1},
\gamma)$. Each cluster contains $V_{n+1}^\gamma$ intervals, hence
$$
\#\{J^+\in \vR^*(\lambda, V_{n+1}, \gamma): J^+\cap I\neq\emptyset\} \ll
|I|V_{n+1}^{\gamma+2}.
$$
Therefore
$$
\frac{\mu(I\cap J)}{\mu(J)}\ll\frac{|I|V_{n+1}^{\gamma+2}}{|J|V_{n+1}^{\gamma+2}}=\frac{|I|}{|J|}.
$$
Hence, for some constant $C_1\ge1$ independent of $n$,
$$
\frac{\mu(I\cap J)}{\mu(J)}\le C_1\left(\frac{|I|}{|J|}\right)^{s'}.
$$

{\bf Case 2.} Let $L_{n+1}\le |I|<V_{n+1}^{-2}$. In this case, $I$
intersects $O(1)$ clusters. Therefore
$$
\#\{J^+\in \vR^*(\lambda, V_{n+1}, \gamma): J^+\cap I\neq\emptyset\} \ll
V_{n+1}^\gamma.
$$
Therefore
$$
\frac{\mu(I\cap J)}{\mu(J)} \ll
\frac{V_{n+1}^{\gamma}}{|J|V_{n+1}^{\gamma+2}} =
\frac{1}{|J|V_{n+1}^2}.
$$
Since $s'<\frac{2}{2+\gamma}$, by choosing $V_{n+1}$ sufficiently large in comparison with $V_n$, we obtain
$$
|J|^{s'-1}V_{n+1}^{-2}\ll V_{n+1}^{-s'(2+\gamma)}=L_{n+1}^{s'}.
$$
Consequently, for some constant $C_2\ge1$ independent of $n$,
$$
\frac{\mu(I\cap J)}{\mu(J)}\le C_2\left(\frac{L_{n+1}}{|J|}\right)^{s'} \le C_2\left(\frac{|I|}{|J|}\right)^{s'}.
$$

{\bf Case 3.} $\rho_{n+1}\le |I|<L_{n+1}$. In this case, $I$
intersects only one cluster. As we have shown, the centres of
intervals in the cluster are separated by a distance $\asymp
V_{n+1}^{-2\gamma - 2}$. Therefore
$$
\#\{J^+\in \vR^*(\lambda, V_{n+1}, \gamma): J^+\cap I\neq\emptyset\} \ll
|I|V_{n+1}^{2\gamma+2} + 1.
$$
Hence
$$
\frac{\mu(I\cap J)}{\mu(J)} \ll \frac{|I|V_{n+1}^\gamma}{|J|} +
\frac{1}{|J|V_{n+1}^{2+\gamma}}.
$$
We estimate the two terms separately. For the first term, since $s'<\frac{2}{2+\gamma}$,  by choosing $V_{n+1}$ sufficiently large in comparison with $V_n$,
we have
$$
|J|^{s'-1}V_{n+1}^{\gamma} \ll V_{n+1}^{(2+\gamma)(1-s')} = L_{n+1}^{s'-1}.
$$
Since $|I|<L_{n+1}$ and $s'-1<0$, we have
$$
L_{n+1}^{s'-1}\le |I|^{s'-1}.
$$
Therefore
$$
\frac{|I|V_{n+1}^{\gamma}}{|J|} \ll \left(\frac{|I|}{|J|}\right)^{s'}.
$$

For the second term, since $s'<\frac{\gamma+2}{(\gamma+3)(1+\lambda)}$,
by choosing $V_{n+1}$ sufficiently large in comparison with $V_n$, we have
$$
|J|^{s'-1}V_{n+1}^{-2-\gamma} \ll V_{n+1}^{-s'(\gamma+3)(1+\lambda)}= \rho_{n+1}^{s'}.
$$
Therefore
$$
\frac{1}{|J|V_{n+1}^{2+\gamma}} \ll \left(\frac{\rho_{n+1}}{|J|}\right)^{s'} \le \left(\frac{|I|}{|J|}\right)^{s'},
$$
where in the last inequality we used $\rho_{n+1}\le|I|$.
Combining the two terms, we obtain
$$
\frac{\mu(I\cap J)}{\mu(J)}\ll \left(\frac{|I|}{|J|}\right)^{s'}.
$$

Finally, combining the three cases, we obtain, for some constant
$C>0$ independent of $n$,
\begin{equation}\label{eq1}
\frac{\mu(I\cap J)}{\mu(J)}
\le C\left(\frac{|I|}{|J|}\right)^{s'},
\qquad
s'<\min\left\{\frac{2}{2+\gamma},
\frac{\gamma+2}{(\gamma+3)(1+\lambda)}\right\}.
\end{equation}
Iterating~\eqref{eq1} along the nested construction and summing over
the $O(1)$ level-$n$ intervals intersecting $I$, we obtain
$$
\mu(I)\ll C^n|I|^{s'}.
$$
Fix $s<s'$. Since the sequence $(V_n)$ may be chosen sufficiently
rapidly increasing, we may assume that $C^n\rho_n^{s'-s}\ll1$.
As $|I|<\rho_n$, it follows that $\mu(I)\ll |I|^s$. Hence, by the mass distribution principle, we get $\dim\KKK(\lambda,\gamma)\ge s$.
Letting $s<s'$ tend to
$$
\min\left\{\frac{2}{2+\gamma},
\frac{\gamma+2}{(\gamma+3)(1+\lambda)}\right\}
$$
gives
$$
\dim\KKK(\lambda,\gamma)\ge
\min\left\{\frac{2}{2+\gamma},
\frac{\gamma+2}{(\gamma+3)(1+\lambda)}\right\}.
$$
Notice that the first term in the
minimum monotonically decreases with $\gamma$ while the second one
monotonically increases. Also notice that for $\frac35\le\lambda<1$
and $\gamma=\alpha = \frac{3\lambda-1}{1-\lambda}$, the first term
is not larger than the second one. Therefore, for a fixed $\frac35\le\lambda<1$, the minimum in~\eqref{eq1} is maximised for
$\gamma=\alpha$ and it equals $\frac{2-2\lambda}{1+\lambda}$.
Finally, this gives
$$
\dim S_s(\lambda) \ge \dim \KKK(\lambda,
\alpha) \ge \frac{2-2\lambda}{1+\lambda}.
$$
Notice that this inequality is also true for $\lambda\ge1$, but in this case it becomes trivial.

The points $\vq_1$ of the form $(v^3,v^2u,vu^2,u^3)$ belong to $\VVV_3$ and therefore clearly belong to $\PPP_s(\lambda)$. It was shown in~\cite{bad_bug_2020} that the set of $\xi$ that lie in infinitely many neighbourhoods $\RRR(\vq_1)$ of such points has Hausdorff dimension $\frac{2}{3(1+\lambda)}$. We deduce that
$$
\dim S_s(\lambda) \ge \frac{2}{3(1+\lambda)}.
$$
The last two inequalities confirm the lower bound of Theorem~\ref{th3}.

\section{Special points: upper bound for the Hausdorff dimension}\label{sec:special-upper}

It is enough to prove the required upper bound on
$S_s(\lambda)\cap I$ for an arbitrary compact interval
$I\Subset\RR$. We fix such an $I$ throughout this section and
suppress it from the notation.

As was mentioned before, the lattice of integer points on
$L_{P_{u,v}}$ is generated by two vectors $\vq_1$ and $\vq_2$ which are defined by~\eqref{def_p12}.
Consider $\vp = \vp(k_1,k_2) := k_1\vq_1 + k_2\vq_2$. For convenience, denote $D := k_1v + k_2r_0$. One can easily compute that
$$
\vp = \left(v^2D, v(uD + k_2), u^2D + 2k_2u, \frac{u^3D + 3k_2u^2}{v}\right).
$$
Since $p_0=v^2D>0$, we have $D>0$.

Let $\PPP_s(\lambda) = \PPP_s^1(\lambda)\cup \PPP_s^2(\lambda)$ where $\PPP_s^1$ contains special vectors with $k_2\neq 0$ or, equivalently, $\Hess(F_{\vp})\not\equiv0$, and $\PPP_s^2$ comes from $k_2=0$ or, equivalently, $\Disc(F_\vp)=0$ and $\Hess(F_\vp)\equiv0$. Then the set $S_s(\lambda)$ splits correspondingly: $S_s(\lambda) = S_s^1(\lambda) \cup S_s^2(\lambda)$.

First assume that $\frac35\le\lambda<1$ and consider
$\vp\in \PPP_s^1(\lambda)$ with $\RRR(\vp)\cap I\ne\emptyset$.
Then Lemma~\ref{lem:approx-window} gives
$$
|p_1^2 - p_0p_2| \ll p_0^{1-\lambda} \quad \Longleftrightarrow\quad
|k_2^2v^2|\ll (v^2D)^{1-\lambda} \quad \Longleftrightarrow\quad (|k_2|v)^{\frac{2}{1-\lambda}}\ll v^2 D.
$$
%Notice that if $k_2r_0 \gg D$ then the last inequality transforms to
% $$
% (k_2v^2r_0)^{1-\lambda} \gg k_2^2v^2\quad\Longleftrightarrow\quad r_0^{1-\lambda}\gg v^{2\lambda} k_2^{1+2\lambda}
% $$
% which is not possible due to $r_0\le v$. Therefore,
% the term $k_1v$ dominates in $D$. Then we have
% $$
% k_2^2v^2 \ll (k_1v^3)^{1-\lambda}\quad \Longleftrightarrow\quad k_1
% \gg k_2^{\frac{2}{1-\lambda}}v^{\frac{3\lambda-1}{1-\lambda}}.
% $$
We now compute
$$
\left|\frac{p_1}{p_0} - \frac{u}{v}\right| \ll \frac{|k_2|v}{v^2D}\ll (|k_2|v)^{-\frac{1+\lambda}{1-\lambda}}.
$$
Hence, for all points $\vp\in \PPP^1_s(\lambda)\cap L_{P_{u,v}}$ and $\xi\in\RRR(\vp)\cap I$,
$$
|v\xi - u|\ll v\left|\xi - \frac{p_1}{p_0}\right| + \left|v\frac{p_1}{p_0}-u\right|\ll
v^{1-\frac{1+\lambda}{1-\lambda}}.
$$
A fixed reduced fraction $u/v$ can occur only finitely often; otherwise $\xi=u/v$, contradicting $|p_0\xi-p_1|=|k_2|v\ge1$.
Hence Jarník's theorem gives
$$
\dim S^1_s(\lambda) \le \frac{2(1-\lambda)}{1+\lambda}.
$$
If $\lambda>1$, one can note that
$$
S_s^1(\lambda)=\emptyset,
$$
while for $\lambda=1$, $\dim S_s^1(1)=0$, so $S_s^1$ does not contribute nontrivially to the dimension of $S_s$ for $\lambda\ge1$.

Now consider the case $k_2=0$. In this case, $\vp=k_1 \vq_1 = k_1(v^3,v^2u,vu^2,u^3)$, and therefore $\vp\in\VVV_3$. This case is established, for example, in~\cite{schleischitz_2016}. But since the arguments are short, we present the proof here.

Notice that $p_0=k_1v^3$ and $\frac{p_1}{p_0}=\frac{u}{v}$. Suppose that $\vp\in \PPP_s^2(\lambda)$ or, in other words, there exists $\xi\in I$ such that
$$
|p_0\xi-p_1|<p_0^{-\lambda}.
$$
Dividing by $p_0$, we obtain
$$
\left|\xi-\frac{u}{v}\right|=\left|\xi-\frac{p_1}{p_0}\right|<p_0^{-1-\lambda}=(k_1v^3)^{-1-\lambda}\le v^{-3(1+\lambda)}.
$$
Discarding a countable set of rational $\xi$, infinitely many distinct
reduced fractions $u/v$ occur, and hence Jarník's theorem gives
$$
\dim S_s^2(\lambda)\le \frac{2}{3(1+\lambda)}.
$$
Combining the two cases, we get
$$
\dim S_{s}(\lambda) \le \max\left\{\frac{2-2\lambda}{1+\lambda},
\frac{2}{3(1+\lambda)}\right\}.
$$
Thus, Theorem~\ref{th3} is proved.

\bibliographystyle{abbrv}
%%%
\bibliography{bibliog}

% \bigskip
% \noindent Dzmitry Badziahin\\ \noindent The University of Sydney\\
% \noindent Camperdown 2006, NSW (Australia)\\
% \noindent {\tt dzmitry.badziahin@sydney.edu.au}

\vspace{2em}

\noindent
\begin{tabular}{@{}l}
  Dmitry Badziahin \\
  The University of Sydney \\
  Camperdown 2006, NSW (Australia) \\
  \texttt{dzmitry.badziahin@sydney.edu.au}
\end{tabular}

\vspace{1.5em}

\noindent
\begin{tabular}{@{}l}
  Nikita Shulga \\
  SMRI, The University of Sydney \\
  Camperdown 2006, NSW (Australia) \\
  \texttt{nikita.shulga@sydney.edu.au}
\end{tabular}

\end{document}